\documentclass[12pt,a4paper]{article}%

\usepackage{amssymb}
\usepackage{amsfonts}
\usepackage{amsmath}
\usepackage{graphicx}%
\providecommand{\U}[1]{\protect\rule{.1in}{.1in}}

\newtheorem{theorem}{Theorem}

\newtheorem{definition}[theorem]{Definition}

\newtheorem{lemma}[theorem]{Lemma}
\newtheorem{notation}[theorem]{Notation}

\newtheorem{proposition}[theorem]{Proposition}
\newtheorem{remark}[theorem]{Remark}

\newenvironment{proof}[1][Proof]{\noindent\textbf{#1.} }{\ \rule{0.5em}{0.5em}}
\makeatletter
\renewcommand\@makefnmark{\relax}
\makeatother
\begin{document}

\title{On the Ornstein-Zernike behaviour for the Bernoulli bond percolation on
$\mathbb{Z}^{d},d\geq3,$ in the supercitical regime}
\author{M. Campanino$^{\#}$, M. Gianfelice$%
{{}^\circ}%
$\\$^{\#}$Dipartimento di Matematica\\Universit\`{a} degli Studi di Bologna\\P.zza di Porta San Donato, 5 \ I-40127\\campanin@dm.unibo.it\\$^{%
{{}^\circ}%
}$Dipartimento di Matematica\\Universit\`{a} della Calabria\\Campus di Arcavacata\\Ponte P. Bucci - cubo 30B\\I-87036 Arcavacata di Rende\\gianfelice@mat.unical.it}
\maketitle

\begin{abstract}
We prove Ornstein-Zernike behaviour in every direction for finite connection
functions of bond percolation on $\mathbb{Z}^{d}$ for $d\geq3$ when $p,$ the
probability of occupation of a bond, is sufficiently close to $1.$ Moreover,
we prove that equi-decay surfaces are locally analytic, strictly convex, with
positive Gaussian curvature.

\end{abstract}

\footnotetext{
\par
\ \hspace*{0.2cm}\emph{Keywords and phrases}: Percolation, local limit
theorem, decay of connectivities, multidimensional renewal process.}

\section{Introduction and results}

Ornstein-Zernike behaviour of correlation and connection functions has been
rigorously proved for many models of statistical mechanics and percolation in
the high temperature or low probability regime, first for extreme values of
the parameter (see e.g. \cite{BF}) and then up to the critical point (see
\cite{CCC}, \cite{CI}, \cite{CIV}). Above the critical probability in
\cite{CIL} it was proved that in two dimensions finite connection functions,
i.e. the probabilities that two sites belong to the same finite cluster,
exhibit a different asymptotic behaviour, which is related to the probability
that two independent random walks in dimension $2$ do not intersect. In higher
dimensions, for $d\geq3,$ one expects that Ornstein-Zernike behaviour holds
for finite connections probabilities above critical probability. This was
proved in \cite{BPS} when $p$ (the probability that a bond is open) is close
to one for the connection probabilities in the direction of the axes. The
proof is based on cluster expansion.

The problem of the asymptotic behaviour of finite connection functions in
arbitrary directions presents an important difference with respect to that in
the directions of coordinate axes. Indeed, in the limit of $p$ tending to $1$
the probability distribution of the finite cluster containing two sites on a
coordinate axis, conditioned to its existence, tends to a delta measure
concentrated on the segment joining the two sites. The cluster expansion
presented in \cite{BPS} can be thought of as a perturbation about this
configuration that plays the role of ground state. In the case of two sites
that don't lie on the same coordinate axis, the limiting distribution is not
supported on a single configuration: one can say that the ground state is
degenerate. This makes the extension of the method used in \cite{BPS} problematic.

In this paper we prove Ornstein-Zernike behaviour in every direction for
finite connection functions of bond percolation on $\mathbb{Z}^{d}$ for
$d\geq3$ when $p,$ the probability of occupation of a bond, is sufficiently
close to $1.$ Moreover, we prove that equi-decay surfaces are locally
analytic, strictly convex with positive Gaussian curvature.

Our proofs rely in part on the methods developed in \cite{CI}, based on
multi-dimensional renewal theory and local limit theorem, but we have to deal
with new problems; in particular FKG inequality does not apply to finite
connection functions, as they are probabilities of non-monotone events. By
developing specific techniques we are able to treat the case when $p$ is
sufficiently close to $1.$

Here in the following are the main results of the paper and the notation that
we will use. In Section 2 we introduce the relevant connectivity functions and
their renewal structure. In Section 3 we prove the existence of the mass-gap
for the direct connectivity function and prove the main results of the paper.

\begin{theorem}
\label{thm1}For any $d\geq3,$ there exists $p^{\ast}=p^{\ast}\left(  d\right)
\in\left(  0,1\right)  $ such that, $\forall p\in\left(  p^{\ast},1\right)  ,$
uniformly in $x\in\mathbb{Z}^{d},\ \left\Vert x\right\Vert \rightarrow\infty,$%
\begin{equation}
\mathbb{P}_{p}\left(  0\longleftrightarrow x\ ,\ \left\vert \mathbf{C}%
_{\{0,x\}}\right\vert <\infty\right)  =\frac{\Phi_{p}\left(  \hat{x}\right)
}{\sqrt{\left(  2\pi\left\Vert x\right\Vert \right)  ^{d-1}}}e^{-\tau
_{p}\left(  x\right)  }\left(  1+o\left(  1\right)  \right)  \ ,
\end{equation}
where $\hat{x}:=\frac{x}{\left\Vert x\right\Vert },\ \Phi_{p}$ is a positive
real analytic function on $\mathbb{S}^{d-1}$ and $\tau_{p}$ an equivalent norm
in $\mathbb{R}^{d}.$
\end{theorem}

As a by-product of the proof of the previous theorem we also obtain the
following result.

\begin{theorem}
For any $d\geq3,$ there exists $p^{\ast}=p^{\ast}\left(  d\right)  \in\left(
0,1\right)  $ such that, $\forall p\in\left(  p^{\ast},1\right)  ,$ the
equi-decay set is locally analytic and strictly convex. Moreover, the Gaussian
curvature of the equi-decay set is uniformly positive.
\end{theorem}

\subsection{Notation}

For any $x\in\mathbb{R}^{d},\ d\geq1,$ Let us denote by $\left\vert
x\right\vert :=\sum_{i=1}^{d}\left\vert x_{i}\right\vert ,$ by $\left\langle
\cdot,\cdot\right\rangle $ the scalar product in $\mathbb{R}^{d}$ and by
$\left\Vert \cdot\right\Vert :=\sqrt{\left\langle \cdot,\cdot\right\rangle }$
the associated Euclidean norm. We then set $\mathbb{S}^{d-1}:=\{z\in
\mathbb{R}^{d}:\left\Vert z\right\Vert =1\}$ and $\hat{x}:=\frac{x}{\left\Vert
x\right\Vert }.$ Given a set $\mathcal{A}\subset\mathbb{R}^{d},$ let us denote
by $\mathcal{A}^{c}$ its complement and by $\mathcal{P}\left(  \mathcal{A}%
\right)  $ the collection of all subsets of $\mathcal{A}.$ We also set
$\mathcal{P}_{2}\left(  \mathcal{A}\right)  :=\{A\in\mathcal{P}\left(
\mathcal{A}\right)  :\left\vert A\right\vert =2\},$ where $\left\vert
A\right\vert $ is the cardinality of $A.$ Moreover, we denote by
$\mathcal{\mathring{A}},\overline{\mathcal{A}}$ respectively the interior of
$\mathcal{A}$ and the closure of $\mathcal{A}$ and set $\mathfrak{d}%
\mathcal{A}:=\overline{\mathcal{A}}\backslash\mathcal{\mathring{A}}$ the
boundary of $\mathcal{A}$ in the Euclidean topology. Furthermore, if
$x\in\mathbb{R}^{d}, $ we set
\begin{equation}
x+\mathcal{A}:=\left\{  y\in\mathbb{R}^{d}:y-x\in\mathcal{A}\right\}
\end{equation}
and, denoting by $B$ the closed unit ball in $\mathbb{R}^{d},$ for $r>0,$ let
$rB:=\left\{  x\in\mathbb{R}^{d}:\left\Vert x\right\Vert \leq r\right\}  ,
$\linebreak$B_{r}\left(  x\right)  :=x+rB.$

To make the paper self-contained, we will now introduce those notions of graph
theory which are going to be used in the sequel and refer the reader to
\cite{B} for an account on this subject.

Let $G=\left(  V,E\right)  $ be a graph whose set of vertices and set of edges
are given respectively by a finite or denumerable set $V$ and $E\subset
\mathcal{P}_{2}\left(  V\right)  .\ G^{\prime}=\left(  V^{\prime},E^{\prime
}\right)  $ such that $V^{\prime}\subseteq V$ and $E^{\prime}\subseteq
\mathcal{P}_{2}\left(  V^{\prime}\right)  \cap E$ is said to be a subgraph of
$G$ and this property is denoted by $G^{\prime}\subseteq G.$ If $G^{\prime
}\subseteq G,$ we denote by $V\left(  G^{\prime}\right)  $ and $E\left(
G^{\prime}\right)  $ respectively the set of vertices and the collection of
the edges of $G^{\prime}.\ \left\vert V\left(  G^{\prime}\right)  \right\vert
$ is called the \emph{order} of $G^{\prime}$ while $\left\vert E\left(
G^{\prime}\right)  \right\vert $ is called its \emph{size}. Given $G_{1}%
,G_{2}\subseteq G,$ we denote by $G_{1}\cup G_{2}:=\left(  V\left(
G_{1}\right)  \cup V\left(  G_{2}\right)  ,E\left(  G_{1}\right)  \cup
E\left(  G_{2}\right)  \right)  \subset G$ the \emph{graph union }of $G_{1}$
and $G_{2}. $ A \emph{path} in $G$ is a subgraph $\gamma$ of $G$ such that
there is a bijection\linebreak$\left\{  0,..,\left\vert E\left(
\gamma\right)  \right\vert \right\}  \ni i\longmapsto v\left(  i\right)
:=x_{i}\in V\left(  \gamma\right)  $ with the property that any $e\in E\left(
\gamma\right)  $ can be represented as $\left\{  x_{i-1},x_{i}\right\}  $ for
$i=1,..,\left\vert E\left(  \gamma\right)  \right\vert .$ A \emph{walk} in $G$
of length $l\geq1$ is an alternating sequence $x_{0},e_{1},x_{1}%
,..,e_{l},x_{l}$ of vertices and edges of $G$ such that $e_{i}=\left\{
x_{i-1},x_{i}\right\}  $ $i=1,..,l.$ Therefore, paths can be associated to
walks having distinct vertices. Two distinct vertices $x,y$ of $G$ are said to
be \emph{connected} if there exists a path $\gamma\subseteq G$ such that
$x_{0}=x,\ x_{\left\vert E\left(  \gamma\right)  \right\vert }=y.$ A graph $G$
is said to be \emph{connected} if any two distinct elements of $V\left(
G\right)  $ are connected. The maximal connected subgraphs of $G$ are called
\emph{components} of $G.$ Given $E^{\prime}\subseteq E,$ we denote by
$G\left(  E^{\prime}\right)  :=\left(  V,E^{\prime}\right)  $ the
\emph{spanning} graph of $E.$ We also define
\begin{equation}
V\left(  E^{\prime}\right)  :=\left(  \bigcup_{e\in E^{\prime}}e\right)
\subset V\ .
\end{equation}
Given $V^{\prime}\subseteq V,$ we set
\begin{equation}
E\left(  V^{\prime}\right)  :=\left\{  e\in E:e\subset V^{\prime}\right\}
\end{equation}
and denote by $G\left[  V^{\prime}\right]  :=\left(  V^{\prime},E\left(
V^{\prime}\right)  \right)  $ that is called the subgraph of $G$
\emph{induced} or \emph{spanned} by $V^{\prime}.$ Moreover, if $G^{\prime
}\subset G,$ we denote by $G\backslash G^{\prime}$ the graph $G\left[
V\backslash V\left(  G^{\prime}\right)  \right]  \subseteq G$ and define the
\emph{boundary} of $G^{\prime}$ as the set
\begin{equation}
\partial G^{\prime}:=\left\{  e\in E\backslash E\left(  G^{\prime}\right)
:\left\vert e\cap V\left(  G^{\prime}\right)  \right\vert =1\right\}  \subset
E\ . \label{defbG}%
\end{equation}

Let $\mathbb{L}^{d}$ be the $d$-dimensional cubic lattice, that is the
geometric graph whose set of vertices is $\mathbb{Z}^{d}$ and whose set of
edges is
\begin{equation}
\mathbb{E}^{d}:=\{\{x,y\}\in\mathcal{P}^{\left(  2\right)  }\left(
\mathbb{Z}^{d}\right)  :\left\vert x-y\right\vert =1\}\ .
\end{equation}
If $G$ is a subgraph of $\mathbb{L}^{d}$ of finite order, we denote by
$\overline{G}$ the graph induced by the union of $V\left(  G\right)  $ with
the the sets of vertices of the connected components of the $\mathbb{L}%
^{d}\backslash G$ of finite size. We define the \emph{external boundary} of
$G$ to be $\overline{\partial}G:=\partial\overline{G}.$ We remark that, given
$G_{i}:=\left(  V_{i},E_{i}\right)  ,\ i=1,2$ two connected subgraphs of
$\mathbb{L}^{d}$ of finite size, by (\ref{defbG}), $\partial\left(  G_{1}\cup
G_{2}\right)  \subseteq\partial G_{1}\cup\partial G_{2}.$ Moreover,
\begin{equation}
\overline{\partial}\left(  G_{1}\cup G_{2}\right)  =\partial\left(
\overline{G_{1}\cup G_{2}}\right)  \subseteq\partial\overline{G_{1}}%
\cup\partial\overline{G_{2}}\ .
\end{equation}
Furthermore, looking at $\mathbb{L}^{d}$ as a cell complex, i.e. as the union
of $\mathbb{Z}^{d}$ and $\mathbb{E}^{d}$ representing respectively the
collection of $0$-cells and of $1$-cells, we denote by $\left(  \mathbb{Z}%
^{d}\right)  ^{\ast}$ the collection of $d$-cells dual $0$-cells in
$\mathbb{L}^{d},$\ that is the collection of unit $d$-cubes centered in the
elements of $\mathbb{Z}^{d}$ (Voronoi cells of $\mathbb{L}^{d}$), and by
$\left(  \mathbb{E}^{d}\right)  ^{\ast}$ the collection of $\left(
d-1\right)  $-cells dual $1$-cells in $\mathbb{L}^{d},$ usually called
\emph{plaquettes} in the physics literature. We also define
\begin{equation}
\mathfrak{E}:=\left\{  \left\{  e_{1}^{\ast},e_{2}^{\ast}\right\}
\in\mathcal{P}^{\left(  2\right)  }\left(  \left(  \mathbb{E}^{d}\right)
^{\ast}\right)  :\mathrm{codim}\left(  \mathfrak{d}e_{1}^{\ast}\cap
\mathfrak{d}e_{2}^{\ast}\right)  =2\right\}
\end{equation}
and consider the graph $\mathfrak{G}:=\left(  \left(  \mathbb{E}^{d}\right)
^{\ast},\mathfrak{E}\right)  .$

A bond percolation configuration on $\mathbb{L}^{d}$ is a map $\mathbb{E}%
^{d}\ni e\longmapsto\omega_{e}\in\{0,1\}.$ Setting\linebreak$\Omega
:=\{0,1\}^{\mathbb{E}^{d}}$ we define
\begin{equation}
\Omega\ni\omega\longmapsto E\left(  \omega\right)  :=\left\{  e\in
\mathbb{E}^{d}:\omega_{e}=1\right\}  \in\mathcal{P}\left(  \mathbb{E}%
^{d}\right)  \ , \label{wE}%
\end{equation}

Let $\mathcal{F}$ be the $\sigma$-algebra generated by the cylinder events of
$\Omega.$ Given $p\in\left[  0,1\right]  ,$ we consider the independent
Bernoulli probability measure $\mathbb{P}_{p}$ on $\left(  \Omega
,\mathcal{F}\right)  $ with parameter $p.$

Denoting by $\mathbb{G}^{d}:=\left\{  G\subseteq\mathbb{L}^{d}:G=G\left(
E\right)  \ ,\ E\in\mathcal{P}\left(  \mathbb{E}^{d}\right)  \right\}  $ the
collection of spanning subgraphs of $\mathbb{L}^{d},$\ we define the random
graph
\begin{equation}
\Omega\ni\omega\longmapsto G\left(  \omega\right)  :=G\left(  E\left(
\omega\right)  \right)  \in\mathbb{G}^{d}\ . \label{wG}%
\end{equation}

Then, given $l\geq1,\ x_{1},..,x_{l}\in\mathbb{Z}^{d},$ we denote by
\begin{equation}
\Omega\ni\omega\longmapsto\mathbf{C}_{\left\{  x_{1},..,x_{l}\right\}
}\left(  \omega\right)  \in\mathcal{P}\left(  \mathbb{Z}^{d}\right)
\end{equation}
the \emph{common open cluster of the points} $x_{1},..,x_{l}\in\mathbb{Z}%
^{d},$ that is the set of vertices of the connected component of the random
graph $G\left(  \omega\right)  $ to which these points belong, provided it
exists, and define, in the case $\mathbf{C}_{\left\{  x_{1},..,x_{l}\right\}
} $ is finite, the random set $\overline{\partial}\mathbf{C}_{\{x_{1}%
,..,x_{l}\}}$ to be equal to $\overline{\partial}G$ if $G$ is the component of
$G\left(  \omega\right)  $ whose set of vertices is $\mathbf{C}_{\{x_{1}%
,..,x_{l}\}}$ and the random set\linebreak$\mathbf{S}_{\{x_{1},..,x_{l}%
\}}:=\left(  \overline{\partial}\mathbf{C}_{\{x_{1},..,x_{l}\}}\right)
^{\ast}.$

\section{Analysis of connectivities}

Given $x,y\in\mathbb{Z}^{d},$ we set
\begin{align}
\varphi\left(  x,y\right)   &  :=\left\{
\begin{array}
[c]{ll}%
\min\left\{  \left\vert \overline{\partial}\mathbf{C}_{\{x,y\}}\left(
\omega\right)  \right\vert :\omega\in\left\{  0<\left\vert \mathbf{C}%
_{\{x,y\}}\left(  \omega\right)  \right\vert <\infty\right\}  \right\}  &
x\neq y\\
0 & x=y
\end{array}
\right. \label{fi}\\
&  =\left\{
\begin{array}
[c]{ll}%
\min\left\{  \left\vert \mathbf{S}_{\{x,y\}}\left(  \omega\right)  \right\vert
:\omega\in\left\{  0<\left\vert \mathbf{C}_{\{x,y\}}\left(  \omega\right)
\right\vert <\infty\right\}  \right\}  & x\neq y\\
0 & x=y
\end{array}
\right.  \ .\nonumber
\end{align}
$\varphi$ is symmetric and translation invariant, therefore in the sequel we
will write
\begin{equation}
\varphi\left(  x,y\right)  =\varphi\left(  x-y\right)  \ .
\end{equation}

For any $x\in\mathbb{Z}^{d}$ and $k\geq\varphi\left(  x\right)  ,$ let us set
\begin{equation}
\mathbf{A}_{k}\left(  x\right)  :=\left\{  \omega\in\Omega:\left\vert
\overline{\partial}\mathbf{C}_{\{0,x\}}\left(  \omega\right)  \right\vert
=k\right\}  =\left\{  \omega\in\Omega:\left\vert \mathbf{S}_{\{0,x\}}\left(
\omega\right)  \right\vert =k\right\}
\end{equation}
and $\mathbf{A}^{k}\left(  x\right)  :=%
{\textstyle\bigvee_{l\geq k}}
\mathbf{A}_{l}\left(  x\right)  .$ We define
\begin{align}
\psi_{k}\left(  x\right)   &  :=\min\left\{  \left\vert E\left(
\mathbf{C}_{\{0,x\}}\left(  \omega\right)  \right)  \right\vert :\omega
\in\mathbf{A}_{k}\left(  x\right)  \right\}  \ ,\label{psikx}\\
\Psi_{k}\left(  x\right)   &  :=\min\left\{  \left\vert E\left(
\mathbf{C}_{\{0,x\}}\left(  \omega\right)  \right)  \right\vert :\omega
\in\mathbf{A}^{k}\left(  x\right)  \right\}  =\min_{l\geq k}\psi_{l}\left(
x\right)  \label{Psikx}%
\end{align}
and set $\mathbf{A}\left(  x\right)  :=\mathbf{A}_{\varphi\left(  x\right)
}\left(  x\right)  $ and consequently $\psi\left(  x\right)  :=\psi
_{\varphi\left(  x\right)  },\Psi\left(  x\right)  :=\Psi_{\varphi\left(
x\right)  }\left(  x\right)  .$

\begin{remark}
Given $x\in\mathbb{Z}^{d},$ for $i=0,..,d,$ let $\gamma_{i}$ be the path such
that $\gamma_{0}=\varnothing$ and
\begin{equation}
V\left(  \gamma_{i}\right)  =\left\{  \operatorname{sign}\left(  x_{i}\right)
u_{i},2\operatorname{sign}\left(  x_{i}\right)  u_{i},..,\left\vert
x_{i}\right\vert \operatorname{sign}\left(  x_{i}\right)  u_{i}\right\}
\qquad i=1,..,d\ .
\end{equation}
Let also
\begin{equation}
\gamma:=\bigcup_{i=0}^{d-1}\left[  \left(  \sum_{j=0}^{i}x_{j}u_{j}\right)
+\gamma_{i+1}\right]  \ .
\end{equation}
By construction
\begin{align}
\left\vert \gamma\right\vert  &  :=\left\vert E\left(  \gamma\right)
\right\vert =\left\vert x\right\vert \ ,\\
\left\vert \overline{\partial}V\left(  \gamma\right)  \right\vert  &
=\left\vert \partial V\left(  \gamma\right)  \right\vert =2\left(  d-1\right)
\left(  \left\vert x\right\vert +1\right)  +2\ .
\end{align}
Hence, by (\ref{psikx}), (\ref{Psikx}) and (\ref{fi}),
\begin{align}
\psi\left(  x\right)   &  \geq\Psi\left(  x\right)  =\left\vert x\right\vert
=\min\left\{  \left\vert \omega\right\vert :\omega\in\{0\longleftrightarrow
x\}\right\}  \ ,\\
\varphi\left(  x\right)   &  \leq2\left(  d-1\right)  \left(  \left\vert
x\right\vert +1\right)  +2\leq2\left(  d-1\right)  \left(  \psi\left(
x\right)  +1\right)  +2\ . \label{fi<psi}%
\end{align}

\end{remark}

\begin{lemma}
\label{sfi}For any $x,y\in\mathbb{Z}^{d},$%
\begin{equation}
\varphi\left(  x\right)  \leq\varphi\left(  y\right)  +\varphi\left(
x-y\right)  \ . \label{gs1}%
\end{equation}

\end{lemma}

\begin{proof}
Given $x,y\in\mathbb{Z}^{d},$ let
\begin{align}
\omega_{1}  &  \in\left\{  \omega\in\Omega:\left\vert \overline{\partial
}\mathbf{C}_{\left\{  y,x\right\}  }\right\vert =\varphi\left(  x-y\right)
\right\}  \ ,\\
\omega_{2}  &  \in\left\{  \omega\in\Omega:\left\vert \overline{\partial
}\mathbf{C}_{\left\{  0,y\right\}  }\right\vert =\varphi\left(  y\right)
\right\}  \ .
\end{align}
there exists $\omega_{3}\in\left\{  0<\left\vert \mathbf{C}_{\{0,x\}}\left(
\omega\right)  \right\vert <\infty\right\}  $ such that $\mathbf{C}_{\left\{
0,x\right\}  }\left(  \omega_{3}\right)  =\mathbf{C}_{\left\{  0,y\right\}
}\left(  \omega_{2}\right)  \cup\mathbf{C}_{\left\{  y,x\right\}  }\left(
\omega_{1}\right)  .$ Moreover, $\overline{\partial}\mathbf{C}_{\left\{
0,x\right\}  }\subseteq\overline{\partial}\mathbf{C}_{\left\{  0,y\right\}
}\cup\overline{\partial}\mathbf{C}_{\left\{  y,x\right\}  }.$ Thus,
\begin{equation}
\varphi\left(  x\right)  \leq\left\vert \overline{\partial}\mathbf{C}%
_{\left\{  0,x\right\}  }\left(  \omega_{3}\right)  \right\vert \leq\left\vert
\overline{\partial}\mathbf{C}_{\left\{  0,y\right\}  }\left(  \omega
_{2}\right)  \right\vert +\left\vert \overline{\partial}\mathbf{C}_{\left\{
y,x\right\}  }\left(  \omega_{1}\right)  \right\vert =\varphi\left(
x-y\right)  +\varphi\left(  y\right)  \ .
\end{equation}

\end{proof}

\begin{proposition}
\label{deffi-}Let, for any $n\in\mathbb{N},\ \mathbb{R}^{d}\ni x\longmapsto
\bar{\varphi}_{n}\left(  x\right)  :=\frac{\varphi\left(  \left\lfloor
nx\right\rfloor \right)  }{n}\in\mathbb{R}^{+}.$ The sequence $\left\{
\bar{\varphi}_{n}\right\}  _{n\in\mathbb{N}}$ converges pointwise to
$\bar{\varphi} $ which is a convex, homogeneous-of-order-one function on
$\mathbb{R}^{d}.$ Moreover, $\left\{  \bar{\varphi}_{n}\right\}
_{n\in\mathbb{N}}$ converges uniformly on $\mathbb{S}^{d-1}.$
\end{proposition}

\begin{proof}
Given $x\in\mathbb{R}^{d}$ and $m,n\in\mathbb{N},$ by the previous lemma,
\begin{equation}
\varphi\left(  \left\lfloor \left(  n+m\right)  x\right\rfloor \right)
\leq\varphi\left(  \left\lfloor nx\right\rfloor \right)  +\varphi\left(
\left\lfloor \left(  n+m\right)  x\right\rfloor -\left\lfloor nx\right\rfloor
\right)  \ .
\end{equation}
Since $\left\vert \left(  \left\lfloor \left(  n+m\right)  x\right\rfloor
-\left\lfloor nx\right\rfloor \right)  -\left\lfloor mx\right\rfloor
\right\vert \leq d,$ it is possible by adding a finite number of open bonds to
construct from each $\omega\in\left\{  0<\left\vert \mathbf{C}%
_{\{0,\left\lfloor mx\right\rfloor \}}\left(  \omega\right)  \right\vert
<\infty\right\}  $ a%
\begin{equation}
\omega^{\prime}\in\left\{  0<\left\vert \mathbf{C}_{\{0,\left(  \left\lfloor
\left(  n+m\right)  x\right\rfloor -\left\lfloor nx\right\rfloor \right)
\}}\left(  \omega\right)  \right\vert <\infty\right\}  \ .
\end{equation}
Hence, there exists a constant $c_{1}=c_{1}\left(  d\right)  $ such that
\begin{equation}
\left\vert \overline{\partial}\mathbf{C}_{\{0,\left(  \left\lfloor \left(
n+m\right)  x\right\rfloor -\left\lfloor nx\right\rfloor \right)  \}}\left(
\omega^{\prime}\right)  \right\vert \leq\left\vert \overline{\partial
}\mathbf{C}_{\{0,\left\lfloor mx\right\rfloor \}}\left(  \omega\right)
\right\vert +c_{1}.
\end{equation}
Then,
\begin{equation}
\varphi\left(  \left\lfloor \left(  n+m\right)  x\right\rfloor \right)
\leq\varphi\left(  \left\lfloor nx\right\rfloor \right)  +\varphi\left(
\left\lfloor mx\right\rfloor \right)  +c_{1}\ , \label{gs2}%
\end{equation}
which imply the existence and the homogeneity of order one of $\bar{\varphi}.
$ The same argument shows that, for any $x,y\in\mathbb{R}^{d}$ and
$n\in\mathbb{N},$%
\begin{align}
\varphi\left(  \left\lfloor nx\right\rfloor \right)   &  \leq\varphi\left(
\left\lfloor ny\right\rfloor \right)  +\varphi\left(  \left\lfloor
nx\right\rfloor -\left\lfloor ny\right\rfloor \right) \label{gs3}\\
&  \leq\varphi\left(  \left\lfloor ny\right\rfloor \right)  +\varphi\left(
\left\lfloor n\left(  x-y\right)  \right\rfloor \right)  +c_{1}\ .\nonumber
\end{align}
Setting $x=\lambda x_{1}+\left(  1-\lambda\right)  x_{2},y=\lambda x_{1},$
with $x_{1},x_{2}\in\mathbb{R}^{d}$ and $\lambda\in\left(  0,1\right)  ,$
dividing by $n$ and taking the limit $n\rightarrow\infty$ we obtain the
convexity of $\bar{\varphi}.$ Moreover, by (\ref{gs3}) and (\ref{fi<psi}),
$\forall x,y\in\mathbb{R}^{d},n\in\mathbb{N},$ there exists a constant
$c_{1}^{\prime}=c_{1}^{\prime}\left(  d\right)  $ such that
\begin{equation}
\left\vert \bar{\varphi}_{n}\left(  x\right)  -\bar{\varphi}_{n}\left(
y\right)  \right\vert \leq c_{1}^{\prime}\left\Vert x-y\right\Vert \ .
\end{equation}
Hence, the collection $\left\{  \bar{\varphi}_{n}\right\}  _{n\in\mathbb{N}}$
is uniformly equicontinuous which, by the compactness of $\mathbb{S}^{d-1},$
implies that $\left\{  \bar{\varphi}_{n}\right\}  _{n\in\mathbb{N}}$ converges uniformly.
\end{proof}

\begin{lemma}
There exists $c_{2}=c_{2}\left(  d\right)  >1$ such that, for any
$x\in\mathbb{Z}^{d},$%
\begin{equation}
c_{2}^{-1}\leq\frac{\varphi\left(  x\right)  }{\psi\left(  x\right)  }\leq
c_{2}\ , \label{fi/psi}%
\end{equation}

\end{lemma}

\begin{proof}
For any $x\in\mathbb{Z}^{d},$ (\ref{fi<psi}) implies $\varphi\left(  x\right)
\leq\left(  2d+1\right)  \psi\left(  x\right)  .$

On the other hand, for any $x\in\mathbb{Z}^{d},$ let
\begin{equation}
\Omega_{x}:=\left\{  \omega\in\left\{  0<\left\vert \mathbf{C}_{\{0,x\}}%
\left(  \omega\right)  \right\vert <\infty\right\}  :\left\vert \overline
{\partial}\mathbf{C}_{\{0,x\}}\left(  \omega\right)  \right\vert
=\varphi\left(  x\right)  ,\ \left\vert E\left(  \mathbf{C}_{\{0,x\}}\left(
\omega\right)  \right)  \right\vert =\psi\left(  x\right)  \right\}  \ .
\end{equation}
Given $\omega\in\Omega_{x},\ \mathbf{C}_{\{0,x\}}\left(  \omega\right)  $ is
contained, as a subset of $\mathbb{R}^{d},$ in the compact connected set
$\bigcup_{x\in\mathbf{C}_{\{0,x\}}\left(  \omega\right)  }x^{\ast}.$ The
function
\begin{equation}
\mathbb{Z}^{d}\ni x\longmapsto\upsilon\left(  x\right)  :=\min_{\omega
\in\Omega_{x}}\left\vert \bigcup_{x\in\mathbf{C}_{\{0,x\}}\left(
\omega\right)  }x^{\ast}\right\vert \in\mathbb{N}%
\end{equation}
is easily seen to satisfy subadditive type inequalities of the kind
(\ref{gs1}) and (\ref{gs2}). Hence, arguing as in the previous proposition,
the sequence $\left\{  \bar{\upsilon}_{n}\right\}  _{n\in\mathbb{N}},$ where,
$\forall n\in\mathbb{N},$%
\begin{equation}
\mathbb{R}^{d}\ni x\longmapsto\bar{\upsilon}_{n}\left(  x\right)
:=\frac{\upsilon\left(  \left\lfloor nx\right\rfloor \right)  }{n}%
\in\mathbb{R}^{+}\ ,
\end{equation}
converges to $\bar{\upsilon}$ which can be proved to be a convex,
homogeneous-of-order-one function. Both $\bar{\varphi}$ and $\bar{\upsilon}$
are equivalent norms in $\mathbb{R}^{d},$ therefore, there exists a positive
constant $c_{2}^{\prime}=c_{2}^{\prime}\left(  d\right)  $ such that
$\bar{\upsilon}\leq c_{2}^{\prime}\bar{\varphi}.$ Moreover,
\begin{equation}
\lim_{\left\Vert x\right\Vert \rightarrow\infty}\frac{\upsilon\left(
x\right)  }{\left\Vert x\right\Vert }=\bar{\upsilon}\left(  \hat{x}\right)
\;;\;\lim_{\left\Vert x\right\Vert \rightarrow\infty}\frac{\varphi\left(
x\right)  }{\left\Vert x\right\Vert }=\bar{\varphi}\left(  \hat{x}\right)  \ .
\end{equation}
Hence, for any $\varepsilon>0,$ there exists $R_{\varepsilon}>0$ such that,
$\forall x\in\left(  R_{\varepsilon}B\right)  ^{c}\cap\mathbb{Z}^{d}%
,$\linebreak$\frac{\upsilon\left(  x\right)  }{\left\Vert x\right\Vert }%
\leq\bar{\upsilon}\left(  \hat{x}\right)  \left(  1+\varepsilon\right)  .$
Then,
\begin{equation}
\psi\left(  x\right)  \leq\upsilon\left(  x\right)  \leq\left(  1+\varepsilon
\right)  \left\Vert x\right\Vert \bar{\upsilon}\left(  \hat{x}\right)
\leq\left(  1+\varepsilon\right)  \left\Vert x\right\Vert c_{2}^{\prime}%
\bar{\varphi}\left(  \hat{x}\right)  \ .
\end{equation}

\end{proof}

\begin{proposition}
\label{P1}There exists a constant $c_{3}=c_{3}\left(  d\right)  >1$ such that,
for any $p\in\left(  1-\frac{1}{c_{3}},1\right)  $ and any $\delta
>\delta^{\ast},$ with
\begin{gather}
\delta^{\ast}=\delta^{\ast}\left(  p,d\right)  :=\frac{\log\frac{c_{3}\left(
d\right)  }{p^{c_{2}\left(  d\right)  }}}{\log\frac{1}{c_{3}\left(  d\right)
\left(  1-p\right)  }}\ ,\label{d*}\\
\mathbb{P}_{p}\left(  \{\left\vert \overline{\partial}\mathbf{C}%
_{\{0,x\}}\right\vert \geq\left(  1+\delta\right)  \varphi\left(  x\right)
\}|\left\{  \left\vert \mathbf{C}_{\{0,x\}}\right\vert <\infty\right\}
\right)  \leq\frac{1}{1-c_{3}\left(  1-p\right)  }\left[  \frac{c_{3}%
^{1+\delta}\left(  1-p\right)  ^{\delta}}{p^{c_{2}}}\right]  ^{\varphi\left(
x\right)  }\ . \label{l1}%
\end{gather}

\end{proposition}

\begin{proof}
Let $\mathcal{L}_{0}^{d}$ be the collection of subgraphs of $\mathbb{L}^{d}$
of finite order. Setting $\mathbb{G}_{0}^{d}:=\left\{  G\in\mathcal{L}_{0}%
^{d}:\partial G=\overline{\partial}G\right\}  $ and denoting by $\mathbb{G}%
_{c}^{d}$ the collection of connected elements of $\mathbb{G}_{0}^{d},$ we
define, for any $k\geq2d,$ the (possibly empty) collection of lattice's
subset
\begin{equation}
\mathbb{G}_{k}:=\left\{  \mathcal{G}\subset\mathfrak{G}:\mathcal{G}=G\left[
\left(  \partial G^{\prime}\right)  ^{\ast}\right]  \ ,\ G^{\prime}%
\in\mathbb{G}_{c}^{d}\ ;\ \left\vert V\left(  \mathcal{G}\right)  \right\vert
=k\right\}  \ .
\end{equation}
Since
\begin{equation}
\{0<\left\vert \mathbf{C}_{\{0,x\}}\right\vert <\infty\}=%
{\displaystyle\bigvee\limits_{k\geq\varphi\left(  x\right)  }}
\mathbf{A}_{k}\left(  x\right)  \ ,
\end{equation}
then,
\begin{equation}
\mathbb{P}_{p}\left(  \mathbf{A}_{k}\left(  x\right)  \right)  =\left(
1-p\right)  ^{k}\sum_{\mathcal{G}\in\mathbb{G}_{k}}\mathbb{P}_{p}\left\{
\omega\in\Omega:G\left[  \mathbf{S}_{\{0,x\}}\left(  \omega\right)  \right]
=\mathcal{G}\right\}
\end{equation}
and we get
\begin{equation}
\mathbb{P}_{p}\left(  \mathbf{A}_{k}\left(  x\right)  \right)  \leq\left\vert
\mathbb{G}_{k}\right\vert \left(  1-p\right)  ^{k}\ .
\end{equation}
We can choose for each $\mathcal{G}\in\mathbb{G}_{k}$ a minimal spanning tree
$T_{\mathcal{G}}$ and consider the collection of graphs
\begin{equation}
\mathbb{T}_{k}:=\left\{  T_{\mathcal{G}}:\mathcal{G}\in\mathbb{G}_{k}\right\}
\ .
\end{equation}
Since given a connected tree there is a walk passing only twice through any
edge of the graph, there exists a constant $c_{3}=c_{3}\left(  d\right)  >1$
such that $\left\vert \mathbb{G}_{k}\right\vert =c_{3}^{k}.$ Therefore, by
(\ref{fi/psi}),
\begin{align}
\mathbb{P}_{p}\{0  &  <\left\vert \mathbf{C}_{\{0,x\}}\right\vert
<\infty\}=\sum_{k\geq\varphi\left(  x\right)  }\mathbb{P}_{p}\left(
\mathbf{A}_{k}\left(  x\right)  \right)  \geq p^{\psi\left(  x\right)
}\left(  1-p\right)  ^{\varphi\left(  x\right)  }\label{lbP}\\
&  \geq\left[  p^{c_{2}}\left(  1-p\right)  \right]  ^{\varphi\left(
x\right)  }\ .\nonumber
\end{align}
Therefore, $\forall p\in\left(  1-\frac{1}{c_{3}},1\right)  ,$ choosing
$\delta^{\ast}$ as in (\ref{d*}), $\forall\delta>\delta^{\ast},$ we have
\begin{gather}
\mathbb{P}_{p}\left(  \{\left\vert \overline{\partial}\mathbf{C}%
_{\{0,x\}}\right\vert \geq\left(  1+\delta\right)  \varphi\left(  x\right)
\}|\left\{  0<\left\vert \mathbf{C}_{\{0,x\}}\right\vert <\infty\right\}
\right)  \leq\frac{\sum_{k\geq\left(  1+\delta\right)  \varphi\left(
x\right)  }\mathbb{P}_{p}\left(  \mathbf{A}_{k}\left(  x\right)  \right)
}{\left[  p^{c_{2}}\left(  1-p\right)  \right]  ^{\varphi\left(  x\right)  }%
}\leq\\
\frac{\sum_{k\geq\left(  1+\delta\right)  \varphi\left(  x\right)  }c_{3}%
^{k}\left(  1-p\right)  ^{k}}{\left[  p^{c_{2}}\left(  1-p\right)  \right]
^{\varphi\left(  x\right)  }}=\frac{1}{1-c_{3}\left(  1-p\right)  }\left[
\frac{c_{3}^{1+\delta}\left(  1-p\right)  ^{\delta}}{p^{c_{2}}}\right]
^{\varphi\left(  x\right)  }.\nonumber
\end{gather}

\end{proof}

\subsection{Renewal structure of connectivities}

Given $t\in\mathbb{S}^{d-1}$ we define
\begin{equation}
\mathcal{H}_{y}^{t}:=\left\{  x\in\mathbb{R}^{d}:\left\langle t,x\right\rangle
=\left\langle t,y\right\rangle \right\}  \quad y\in\mathbb{R}^{d}%
\end{equation}
to be the $\left(  d-1\right)  $-dimensional hyperplane in $\mathbb{R}^{d}$
orthogonal to the vector $t$ passing through a point $y\in\mathbb{R}^{d}$ and
the corresponding half-spaces
\begin{align}
\mathcal{H}_{y}^{t,-}  &  :=\left\{  x\in\mathbb{R}^{d}:\left\langle
t,x\right\rangle \leq\left\langle t,y\right\rangle \right\}  \ ,\\
\mathcal{H}_{y}^{t,+}  &  :=\left\{  x\in\mathbb{R}^{d}:\left\langle
t,x\right\rangle \geq\left\langle t,y\right\rangle \right\}  \ .
\end{align}
Let $t\in\mathbb{S}^{d}.$ Given two points $x,y\in\mathbb{Z}^{d}$ such that
$\left\langle x,t\right\rangle \leq\left\langle y,t\right\rangle ,$ we denote
by $\mathbf{C}_{\{x,y\}}^{t}$ the cluster of $x$ and $y$ inside the strip
$\mathcal{S}_{\left\{  x,y\right\}  }^{t}:=\mathcal{H}_{x}^{t,+}%
\cap\mathcal{H}_{y}^{t,-}$ provided it exists.

Let $u$ be the first of the unit vectors in the direction of the coordinate
axis $u_{1},..,u_{d}$ such that $\left\langle t,u\right\rangle $ is maximal

\begin{definition}
Given $t\in\mathbb{S}^{d-1},$ let $x,y\in\mathbb{Z}^{d}$ such that
$\left\langle x,t\right\rangle \leq\left\langle y,t\right\rangle ,$ be
connected in $\mathcal{S}_{\left\{  x,y\right\}  }^{t}.$ The points
$b\in\mathbf{C}_{\{x,y\}}^{t}$ such that:

\begin{enumerate}
\item $\left\langle t,x+u\right\rangle \leq\left\langle t,b\right\rangle
\leq\left\langle t,y-u\right\rangle ;$

\item $\mathbf{C}_{\{x,y\}}^{t}\cap\mathcal{S}_{\left\{  b-u,b+u\right\}
}^{t}=\left\{  b-u,b,b+u\right\}  ;$
\end{enumerate}

are said to be $t$\emph{-break points} of $\mathbf{C}_{\{x,y\}}.$ The
collection of such points, which we remark is a totally ordered set with
respect to the scalar product with $t,$ will be denoted by $\mathbf{B}%
^{t}\left(  x,y\right)  .$
\end{definition}

\begin{definition}
Given $t\in\mathbb{S}^{d-1},$ let $x,y\in\mathbb{Z}^{d}$ such that
$\left\langle x,t\right\rangle \leq\left\langle y,t\right\rangle ,$ be
connected in $\mathcal{S}_{\left\{  x,y\right\}  }^{t}.$ An edge $\{b,b+u\}$
such that $b,b+u\in\mathbf{B}^{t}\left(  x,y\right)  $ is called
$t$\emph{-bond} of $\mathbf{C}_{\{x,y\}}.$ The collection of such edges will
be denoted by $\mathbf{E}^{t}\left(  x,y\right)  ,$ while $\mathbf{B}_{e}%
^{t}\left(  x,y\right)  \subset\mathbf{B}^{t}\left(  x,y\right)  $ will denote
the subcollection of $t$-break points\emph{\ }$b$ of $\mathbf{C}_{\{x,y\}}$
such that the edge $\{b,b+u\}\in\mathbf{E}^{t}\left(  x,y\right)  .$
\end{definition}

\begin{notation}
In the sequel we will omit the dependence on $x$ in the notation of the random
sets $\mathbf{B}^{t}\left(  x,y\right)  ,\ \mathbf{B}_{e}^{t}\left(
x,y\right)  $ and $\mathbf{E}^{t}\left(  x,y\right)  $ if such point is taken
to be the origin.
\end{notation}

\begin{definition}
Given $t\in\mathbb{S}^{d-1},$ let $x,y\in\mathbb{Z}^{d}$ such that
$\left\langle t,x\right\rangle \leq\left\langle t,y\right\rangle $ be
connected.$\ x,y\in\mathbb{Z}^{d}$ are said to be $h_{t}$-connected if

\begin{enumerate}
\item $x$ and $y$ are connected in $\mathcal{S}_{\{x,y\}}^{t}$ and $\left\vert
\mathbf{C}_{\{x,y\}}^{t}\right\vert <\infty;$

\item $x+u,y-u\in\mathbf{B}^{t}\left(  x,y\right)  .$
\end{enumerate}
\end{definition}

Moreover, denoting by $\left\{  x\overset{h_{t}}{\longleftrightarrow
y}\right\}  $ the event that $x$ and $y$ are $h_{t}$-connected, we set
\begin{equation}
h_{t}^{\left(  p\right)  }\left(  x,y\right)  :=\mathbb{P}_{p}\left\{
x\overset{h_{t}}{\longleftrightarrow y}\right\}  \ .
\end{equation}

Notice that, by translation invariance, $h_{t}^{\left(  p\right)  }\left(
x,y\right)  =h_{t}^{\left(  p\right)  }\left(  y-x,0\right)  $ so in the
sequel we will denote it simply by $h_{t}^{\left(  p\right)  }\left(
y-x\right)  .$ We also define by convention $h_{t}^{\left(  p\right)  }\left(
0\right)  =1.$

\begin{definition}
Let $t\in\mathbb{S}^{d-1}$ and $x,y\in\mathbb{Z}^{d}$ be $h_{t}$-connected. If
$\mathbf{B}^{t}\left(  x+u,y-u\right)  $ is empty, then $x$ and $y$ are said
to be $f_{t}$-connected and the corresponding event is denoted by $\left\{
x\overset{f_{t}}{\longleftrightarrow y}\right\}  .$ We then set\
\begin{equation}
f_{t}^{\left(  p\right)  }\left(  y-x\right)  :=\mathbb{P}_{p}\left\{
x\overset{f_{t}}{\longleftrightarrow y}\right\}  \ .
\end{equation}

\end{definition}

We define by convention $f_{t}^{\left(  p\right)  }\left(  0\right)  =0.$

\begin{definition}
Given $t\in\mathbb{S}^{d-1},$ let $x,y\in\mathbb{Z}^{d}$ such that
$\left\langle t,x\right\rangle \leq\left\langle t,y\right\rangle $ be
connected. Then:

\begin{enumerate}
\item $x,y$ are called $\bar{h}_{t}$-connected and the corresponding event is
denoted by $\left\{  x\overset{\bar{h}_{t}}{\longleftrightarrow y}\right\}  ,$
if $\mathbf{C}_{\{x,y\}}\cap\mathcal{S}_{\left\{  y-u,y\right\}  }%
^{t}=\left\{  y-u,y\right\}  $ and $\left\vert \mathbf{C}_{\{x,y\}}%
\cap\mathcal{H}_{y}^{t,-}\right\vert <\infty.$

\item $x,y$ are called $\bar{f}_{t}$-connected and the corresponding event is
denoted by $\left\{  x\overset{\bar{f}_{t}}{\longleftrightarrow y}\right\}  ,$
if they are $\bar{h}_{t}$-connected and $\mathbf{B}^{t}\left(  x,y\right)
=\emptyset.$
\end{enumerate}
\end{definition}

\begin{definition}
Given $t\in\mathbb{S}^{d-1},$ let $x,y\in\mathbb{Z}^{d}$ such that
$\left\langle t,x\right\rangle \leq\left\langle t,y\right\rangle $ be
connected. Then:

\begin{enumerate}
\item $x,y$ are called $\tilde{h}_{t}$-connected and the corresponding event
is denoted by $\left\{  x\overset{\tilde{h}_{t}}{\longleftrightarrow
y}\right\}  ,$ if:

\begin{enumerate}
\item $\mathbf{C}_{\left\{  x,y\right\}  }\cap\mathcal{S}_{\{x,x+u\}}%
^{t}=\{x,x+u\};$

\item $\left\vert \mathbf{C}_{\{x,y\}}\cap\mathcal{H}_{x}^{t,+}\right\vert
<\infty.$
\end{enumerate}

\item $x,y$ are called $\tilde{f}_{t}$-connected and the corresponding event
is denoted by $\left\{  x\overset{\tilde{f}_{t}}{\longleftrightarrow
y}\right\}  ,$ if they are $\tilde{h}_{t}$-connected and $\mathbf{B}%
^{t}\left(  x,y\right)  =\emptyset.$
\end{enumerate}
\end{definition}

The functions $\bar{h}_{t}^{\left(  p\right)  }\left(  x,y\right)
:=\mathbb{P}_{p}\left\{  x\overset{\bar{h}_{t}}{\longleftrightarrow
y}\right\}  $ and $\tilde{h}_{t}^{\left(  p\right)  }\left(  x,y\right)
:=\mathbb{P}_{p}\left\{  x\overset{\tilde{h}_{t}}{\longleftrightarrow
y}\right\}  $ are translation invariant.

Denoting by $g_{t}^{\left(  p\right)  }\left(  x,y\right)  ,$ for
$t\in\mathbb{S}^{d-1},$ the probability of the event
\begin{equation}
\left\{  x\overset{g_{t}}{\longleftrightarrow}y\right\}  :=\left\{
x\longleftrightarrow y,\ \left\vert \mathbf{C}_{\{x,y\}}\right\vert
<\infty,~\left\vert \mathbf{B}_{e}^{t}\left(  x,y\right)  \right\vert
\leq1\right\}  \ , \label{gcon}%
\end{equation}
which is also translation invariant, we obtain
\begin{gather}
\mathbb{P}_{p}\{0\longleftrightarrow x\ ,\ \left\vert \mathbf{C}%
_{\{0,x\}}\right\vert <\infty\}=g_{t}^{\left(  p\right)  }\left(  x\right)
+\sum_{z_{1},z_{2}\in\mathbb{Z}^{d}}\bar{f}_{t}^{\left(  p\right)  }\left(
z_{1}\right)  h_{t}^{\left(  p\right)  }\left(  z_{2}-z_{1}\right)  \tilde
{f}_{t}^{\left(  p\right)  }\left(  x-z_{2}\right)  \ ,\label{renh1}\\
h_{t}^{\left(  p\right)  }\left(  x\right)  =\sum_{z\in\mathbb{Z}^{d}}%
f_{t}^{\left(  p\right)  }\left(  z\right)  h_{t}^{\left(  p\right)  }\left(
x-z\right)  \ . \label{renh}%
\end{gather}

\begin{proposition}
Given $t\in\mathbb{S}^{d-1},$ for any $p\in\left(  0,1\right)  $ and
$x\in\mathbb{R}^{d}$ such that $\left\langle t,x\right\rangle >0,$%
\begin{equation}
\tau_{p}^{t}\left(  x\right)  :=-\lim_{n\rightarrow\infty}\frac{1}{n}\log
h_{t}^{\left(  p\right)  }\left(  \left\lfloor nx\right\rfloor \right)
\label{tau}%
\end{equation}
exists and is a convex and homogeneous-of-order-one function on $\mathbb{R}%
^{d}.$ Moreover, for\linebreak$p\in\left(  1-\frac{1}{c_{3}},1\right)  ,$%
\begin{equation}
\tau_{p}^{t}\left(  x\right)  \geq\bar{\varphi}\left(  x\right)  \log\frac
{1}{c_{3}\left(  1-p\right)  }\ . \label{tau>fi}%
\end{equation}

\end{proposition}

\begin{proof}
We proceed as in \cite{ACC}. If $x,y\in\mathbb{Z}^{d}$ are $h_{t}$-connected
and $\mathcal{S}_{\{x+u,y-u\}}^{t}\cap\mathbb{Z}^{d}\neq\varnothing,$ by the
definition of the $h_{t}$-connection, for any $z\in\mathcal{S}_{\{x+u,y-u\}}%
^{t}$ and $p\in\left(  0,1\right)  ,$%
\begin{equation}
h_{t}^{\left(  p\right)  }\left(  y-x\right)  \geq h_{t}^{\left(  p\right)
}\left(  z-x\right)  h_{t}^{\left(  p\right)  }\left(  y-z\right)  \ .
\label{superh}%
\end{equation}
Then, for any $x\in\mathbb{Q}^{d}$ and $n,m\in\mathbb{N},$ if $k\in\mathbb{N}
$ is such that $kx\in\mathbb{Z}^{d},$ we have
\begin{align}
\log h_{t}^{\left(  p\right)  }\left(  \left(  n+m\right)  kx\right)   &
\geq\log h_{t}^{\left(  p\right)  }\left(  nkx\right)  +\log h_{t}^{\left(
p\right)  }\left(  \left(  n+m\right)  kx-nkx\right) \\
&  =\log h_{t}^{\left(  p\right)  }\left(  nkx\right)  +\log h_{t}^{\left(
p\right)  }\left(  mkx\right)  \ ,\nonumber
\end{align}
which imply that, for any $x\in\mathbb{Q}^{d},$ the limit $\tau_{p}^{t}$ in
(\ref{tau}) exists and has the property
\begin{equation}
\tau_{p}^{t}\left(  \lambda x\right)  =\left\vert \lambda\right\vert \tau
_{p}^{t}\left(  x\right)  \ ,\;\lambda\in\mathbb{Q}\ .
\end{equation}
Moreover, given $x,y\in\mathbb{Q}^{d},\lambda\in\mathbb{Q\cap}\left[
0,1\right]  $ and $n\in\mathbb{N}$ such that $n\lambda x\in\mathbb{Z}^{d}%
,$\linebreak$n\left(  1-\lambda\right)  y\in\mathbb{Z}^{d}\cap\mathcal{S}%
_{\{u,n\left(  n\lambda x+n\left(  1-\lambda\right)  y\right)  -u\}}^{t},$ by
(\ref{superh}),
\begin{equation}
\log h_{t}^{\left(  p\right)  }\left(  n\left(  1-\lambda\right)  y+n\lambda
x\right)  \geq\log h_{t}^{\left(  p\right)  }\left(  n\left(  1-\lambda
\right)  y\right)  +\log h_{t}^{\left(  p\right)  }\left(  n\lambda x\right)
\ ,
\end{equation}
which, together with the previous property, implies
\begin{equation}
\tau_{p}^{t}\left(  n\left(  1-\lambda\right)  y+n\lambda x\right)  \leq
\tau_{p}^{t}\left(  n\left(  1-\lambda\right)  y\right)  +\tau_{p}^{t}\left(
n\lambda x\right)  \ .
\end{equation}
Therefore, $\tau_{p}^{t}$ is a convex, hence continuous, function on
$\mathbb{Q}^{d}$ and can be extended to a convex and homogeneous-of-order-one
function on $\mathbb{R}^{d}.$ Furthermore, for $p\in\left(  1-\frac{1}{c_{3}%
},1\right)  ,$ from the inequality
\begin{equation}
h_{t}^{\left(  p\right)  }\left(  \left\lfloor nx\right\rfloor \right)
\leq\mathbb{P}_{p}\{0<\left\vert \mathbf{C}_{\{0,\left\lfloor nx\right\rfloor
\}}\right\vert <\infty\}\leq\frac{\left(  c_{3}\left(  1-p\right)  \right)
^{\varphi\left(  \left\lfloor nx\right\rfloor \right)  }}{1-c_{3}\left(
1-p\right)  }%
\end{equation}
and Proposition \ref{deffi-} we get (\ref{tau>fi}).
\end{proof}

Since for any $p\in\left(  0,1\right)  $ and $d\geq2$ (cf. \cite{G} section
8.5) there exists $c_{-}=c_{-}\left(  p,d\right)  >0$ such that
\begin{equation}
\mathbb{P}_{p}\{0\longleftrightarrow x\ ,\ \left\vert \mathbf{C}%
_{\{0,x\}}\right\vert <\infty\}\leq e^{-c_{-}\left\Vert x\right\Vert }\ ,
\label{ubP}%
\end{equation}
while (\ref{lbP}) implies that there exists $c_{+}=c_{+}\left(  p,d\right)
>0$ such that
\begin{equation}
h_{t}^{\left(  p\right)  }\left(  x\right)  \geq e^{-c_{+}\left\Vert
x\right\Vert }\ . \label{lbh}%
\end{equation}
it follows that $\tau_{p}$ is finite and is an equivalent norm in
$\mathbb{R}^{d}.$

Furthermore, let
\begin{equation}
\mathbb{R}^{d}\ni s\longmapsto H_{t}^{\left(  p\right)  }\left(  s\right)
:=\sum_{x\in\mathbb{Z}^{d}}h_{t}^{\left(  p\right)  }\left(  x\right)
e^{\left\langle s,x\right\rangle }\in\overline{\mathbb{R}}\ .
\end{equation}
(\ref{ubP}) implies that $\forall p\in\left(  p_{c}\left(  d\right)
,1\right)  , $ the effective domain of $H_{t}\left(  s\right)  ,$%
\begin{equation}
\mathcal{D}_{t}^{p}:=\left\{  s\in\mathbb{R}^{d}:H_{t}^{\left(  p\right)
}\left(  s\right)  <\infty\right\}  \ ,
\end{equation}
is not empty since $\mathcal{\mathring{D}}_{t}^{p}\supseteq\mathcal{\mathring
{K}}_{t}^{p}\ni0,$ where
\begin{equation}
\mathcal{K}_{t}^{p}:=\bigcap_{\hat{x}\in\mathbb{S}^{d-1}}\left\{
s\in\mathbb{R}^{d}:\left\langle s,\hat{x}\right\rangle \leq\tau_{t}^{p}\left(
\hat{x}\right)  \right\}
\end{equation}
is the convex body polar with respect to $\mathcal{U}_{t}^{p}:=\left\{
x\in\mathbb{R}^{d}:\tau_{t}^{p}\left(  x\right)  \leq1\right\}  .$

\section{Renormalization}

Let
\begin{equation}
\mathcal{W}:=\bigcap_{\hat{x}\in\mathbb{S}^{d-1}}\left\{  s\in\mathbb{R}%
^{d}:\left\langle s,\hat{x}\right\rangle \leq\bar{\varphi}\left(  \hat
{x}\right)  \right\}
\end{equation}
and, for any $x\in\mathbb{R}^{d},$ let $\mathbb{S}_{x}^{d-1}:=\left\{  \hat
{s}\in\mathbb{S}^{d-1}:s\in\mathcal{W},\ \left\langle s,\hat{x}\right\rangle
=\bar{\varphi}\left(  \hat{x}\right)  \right\}  .$

\begin{definition}
Given $t\in\mathbb{S}^{d-1},$ for any $x,y\in\mathbb{Z}^{d}$ such that
$\left\langle t,x\right\rangle \leq\left\langle t,y\right\rangle ,$ let
\begin{equation}
\left\{  \mathcal{H}_{x}^{t}\longleftrightarrow\mathcal{H}_{y}^{t}\right\}
:=\left\{  \omega\in\Omega:0<\left\vert \mathbf{C}_{\left\{  x^{\prime
},y^{\prime}\right\}  }\right\vert <\infty\ ;\ x^{\prime}\in\mathcal{H}%
_{x}^{t,-}\cap\mathbb{Z}^{d},\ y^{\prime}\in\mathcal{H}_{y}^{t,+}%
\cap\mathbb{Z}^{d}\right\}  \ .
\end{equation}
We define
\begin{equation}
\varphi_{t}\left(  x,y\right)  :=\min\left\{  \left\vert \mathbf{S}_{\left\{
x,y\right\}  }\left(  \omega\right)  \right\vert :\omega\in\left\{
\mathcal{H}_{x}^{t}\longleftrightarrow\mathcal{H}_{y}^{t}\right\}  \right\}
\end{equation}
and
\begin{equation}
\phi_{t}\left(  x,y\right)  :=\min_{\omega\in\left\{  \mathcal{H}_{x}%
^{t}\longleftrightarrow\mathcal{H}_{y}^{t}\right\}  }\left\vert \left\{
e^{\ast}\in\mathbf{S}_{\left\{  x,y\right\}  }\left(  \omega\right)  :e^{\ast
}\subset\mathcal{S}_{\left\{  x,y\right\}  }^{t}\right\}  \right\vert \ .
\end{equation}

\end{definition}

Notice that, by translation invariance, $\phi_{t}\left(  x,y\right)  =\phi
_{t}\left(  y-x,0\right)  $ therefore we are allowed to write $\phi_{t}\left(
x,y\right)  =\phi_{t}\left(  y-x\right)  .$ Moreover,
\begin{equation}
\varphi_{t}\left(  x,y\right)  =\min_{\left(  z^{\prime},y^{\prime}\right)
\in\mathcal{H}_{x}^{t,-}\cap\mathbb{Z}^{d}\times\mathcal{H}_{y}^{t,+}%
\cap\mathbb{Z}^{d}}\varphi\left(  y^{\prime}-z^{\prime}\right)  =\min
_{y^{\prime}\in\mathcal{H}_{y}^{t,+}\cap\mathbb{Z}^{d}}\varphi\left(
y^{\prime}-x\right)  \ .
\end{equation}
Hence we set $\varphi_{t}\left(  y-x\right)  :=\varphi_{t}\left(  x,y\right)
.$ We also remark that Lemma \ref{sfi} and Proposition \ref{deffi-} imply the
existence of $\bar{\varphi}_{t}\left(  x\right)  :=\lim_{n\rightarrow\infty
}\frac{\varphi_{t}\left(  \left\lfloor nx\right\rfloor \right)  }{n}$ for any
$x\in\mathbb{R}^{d}$ such that $\left\langle t,x\right\rangle >0.$

\begin{lemma}
For any $x\in\mathbb{R}^{d}$ and $t\in\mathbb{S}_{x}^{d-1},$ there exists
$\lim_{n\rightarrow\infty}\frac{\phi_{t}\left(  \left\lfloor nx\right\rfloor
\right)  }{n}=\bar{\varphi}\left(  x\right)  .$
\end{lemma}

\begin{proof}
For any $t\in\mathbb{S}^{d-1},\ \phi_{t}$\ is superadditive and consequently
$\bar{\phi}_{t}\left(  x\right)  :=\lim_{n\rightarrow\infty}\frac{\phi
_{t}\left(  \left\lfloor nx\right\rfloor \right)  }{n}$ exists for any
$x\in\mathbb{R}^{d}$ such that $\left\langle t,x\right\rangle >0.$
Furthermore, $\bar{\phi}_{t}$ is a homogeneous-of-order-one function.

Given $t\in\mathbb{S}^{d-1}$ and $x\in\mathbb{Z}^{d}$ such that $\left\langle
t,x\right\rangle >0,$ by the definition of $\phi_{t}$ and $\varphi_{t},$ it
follows that $\phi_{t}\left(  x\right)  \leq\varphi_{t}\left(  x\right)  .$
Moreover, by the convexity of $\bar{\varphi},$ if $t$ and $x$ are chosen to be
in polar relation with respect to $\bar{\varphi},$ we have $\bar{\varphi}%
_{t}\left(  x\right)  =\bar{\varphi}\left(  x\right)  .$ Hence, $\forall
x\in\mathbb{R}^{d}$ and $t\in\mathbb{S}_{x}^{d-1},$\ $\bar{\phi}_{t}\left(
x\right)  \leq\bar{\varphi}\left(  x\right)  .$

On the other hand, for any $\hat{x}\in\mathbb{S}^{d-1}$ and $t\in
\mathbb{S}_{x}^{d-1},$ let us consider the slab $\mathcal{S}_{\left\{
0,\left\lfloor n\hat{x}\right\rfloor \right\}  }^{t},$ with $n$ a large
integer. Given $\delta\in\left(  0,\frac{1}{2\left(  d-1\right)  }\right)
,\ N\in\mathbb{N},$ and $K_{n}^{N}\left(  \delta\right)  :=\left\lfloor
\frac{n^{1-\delta}}{N}\right\rfloor ,$ we can decompose $\mathcal{S}_{\left\{
0,\left\lfloor n\hat{x}\right\rfloor \right\}  }^{t}$ as follows
\begin{equation}
\mathcal{S}_{\left\{  0,\left\lfloor n\hat{x}\right\rfloor \right\}  }%
^{t}=\mathcal{S}_{\left\{  0,\left\lfloor K_{n}^{N}\left(  \delta\right)
N\hat{x}\right\rfloor \right\}  }^{t}\cup\mathcal{S}_{\left\{  \left\lfloor
K_{n}^{N}\left(  \delta\right)  N\hat{x}\right\rfloor ,\left\lfloor n\hat
{x}\right\rfloor -\left\lfloor K_{n}^{N}\left(  \delta\right)  N\hat
{x}\right\rfloor \right\}  }^{t}\cup\mathcal{S}_{\left\{  \left\lfloor
n\hat{x}\right\rfloor -\left\lfloor K_{n}^{N}\left(  \delta\right)  N\hat
{x}\right\rfloor ,\left\lfloor n\hat{x}\right\rfloor \right\}  }^{t}%
\end{equation}
where
\begin{align}
\mathcal{S}_{\left\{  0,\left\lfloor K_{n}^{N}\left(  \delta\right)  N\hat
{x}\right\rfloor \right\}  }^{t}  &  =\bigcup_{i=0,..,K_{n}^{N}\left(
\delta\right)  -1}\mathcal{S}_{\left\{  \left\lfloor iN\hat{x}\right\rfloor
,\left\lfloor \left(  i+1\right)  N\hat{x}\right\rfloor \right\}  }^{t}\\
\mathcal{S}_{\left\{  \left\lfloor n\hat{x}\right\rfloor -\left\lfloor
K_{n}^{N}\left(  \delta\right)  N\hat{x}\right\rfloor ,\left\lfloor n\hat
{x}\right\rfloor \right\}  }^{t}  &  =\bigcup_{i=0,..,K_{n}^{N}\left(
\delta\right)  -1}\mathcal{S}_{\left\{  \left\lfloor n\hat{x}\right\rfloor
-\left\lfloor \left(  i+1\right)  N\hat{x}\right\rfloor ,\left\lfloor n\hat
{x}\right\rfloor -\left\lfloor iN\hat{x}\right\rfloor \right\}  }^{t}\ .
\end{align}
For any configuration in $\left\{  \omega\in\Omega:\left\vert \overline
{\partial}\mathbf{C}_{\left\{  0,\left\lfloor n\hat{x}\right\rfloor \right\}
}\left(  \omega\right)  \right\vert =\varphi_{t}\left(  \left\lfloor n\hat
{x}\right\rfloor \right)  \right\}  ,$ since there exists a constant $c_{+}$
such that $\varphi_{t}\left(  \left\lfloor n\hat{x}\right\rfloor \right)  \leq
c_{+}n,$ there is at least one slab in the collection $\left\{  \mathcal{S}%
_{\left\{  \left\lfloor iN\hat{x}\right\rfloor ,\left\lfloor \left(
i+1\right)  N\hat{x}\right\rfloor \right\}  }^{t}\right\}  _{i\in\left\{
0,..,K_{n}^{N}\left(  \delta\right)  -1\right\}  },$ which we denote by
$\mathcal{S}_{\left\{  \left\lfloor jN\hat{x}\right\rfloor ,\left\lfloor
\left(  j+1\right)  N\hat{x}\right\rfloor \right\}  }^{t},$ such that
\begin{equation}
\left\vert \left\{  e^{\ast}\in\mathbf{S}_{\left\{  0,\left\lfloor n\hat
{x}\right\rfloor \right\}  }\left(  \omega\right)  :e^{\ast}\subset
\mathcal{S}_{\left\{  \left\lfloor jN\hat{x}\right\rfloor ,\left\lfloor
\left(  j+1\right)  N\hat{x}\right\rfloor \right\}  }^{t}\right\}  \right\vert
\leq c_{4}^{\prime}n^{2\delta}\ ,
\end{equation}
with $c_{4}^{\prime}=c_{4}^{\prime}\left(  N\right)  .$ The same argument also
apply to the slabs in the collection\linebreak$\left\{  \mathcal{S}_{\left\{
\left\lfloor n\hat{x}\right\rfloor -\left\lfloor \left(  i+1\right)  N\hat
{x}\right\rfloor ,\left\lfloor n\hat{x}\right\rfloor -\left\lfloor iN\hat
{x}\right\rfloor \right\}  }^{t}\right\}  _{i\in\left\{  0,..,K_{n}^{N}\left(
\delta\right)  -1\right\}  },$ hence there exists $l\in\left\{  0,..,K_{n}%
^{N}\left(  \delta\right)  -1\right\}  $ such that
\begin{equation}
\left\vert \left\{  e^{\ast}\in\mathbf{S}_{\left\{  0,\left\lfloor n\hat
{x}\right\rfloor \right\}  }\left(  \omega\right)  :e^{\ast}\subset
\mathcal{S}_{\left\{  \left\lfloor n\hat{x}\right\rfloor -\left\lfloor \left(
l+1\right)  N\hat{x}\right\rfloor ,\left\lfloor n\hat{x}\right\rfloor
-\left\lfloor lN\hat{x}\right\rfloor \right\}  }^{t}\right\}  \right\vert \leq
c_{4}^{\prime}n^{2\delta}\ .
\end{equation}
Let $\mathbf{C}_{j,l}^{t}\left(  \omega\right)  $ be a connected component of
$\mathbf{C}_{\left\{  0,\left\lfloor n\hat{x}\right\rfloor \right\}  }\left(
\omega\right)  \cap\mathcal{S}_{\left\{  \left\lfloor jN\hat{x}\right\rfloor
,\left\lfloor n\hat{x}\right\rfloor -\left\lfloor lN\hat{x}\right\rfloor
\right\}  }^{t}$ connecting $\mathcal{H}_{\left\lfloor jN\hat{x}\right\rfloor
}^{t}$ with $\mathcal{H}_{\left\lfloor n\hat{x}\right\rfloor -\left\lfloor
lN\hat{x}\right\rfloor }^{t,+}.$ Since the sum of the diameters of the
components of the subgraph of $\mathfrak{G}$ induced by $\left\{  e^{\ast}%
\in\mathbf{S}_{\left\{  0,\left\lfloor n\hat{x}\right\rfloor \right\}
}\left(  \omega\right)  :\left\vert e^{\ast}\cap\mathcal{H}_{y}^{t}\right\vert
>0\right\}  ,$ with $y=\left\lfloor jN\hat{x}\right\rfloor ,\left\lfloor
n\hat{x}\right\rfloor -\left\lfloor lN\hat{x}\right\rfloor ,$ is smaller than
$c_{4}^{\prime}n^{2\delta},$ there exists $c_{4}^{\prime\prime}=c_{4}%
^{\prime\prime}\left(  N\right)  $ such that
\begin{align}
\left\vert \left\{  e^{\ast}\in\left(  \overline{\partial}\mathbf{C}_{j,l}%
^{t}\left(  \omega\right)  \right)  ^{\ast}:\left\vert e^{\ast}\cap
\mathcal{H}_{\left\lfloor jN\hat{x}\right\rfloor }^{t,-}\right\vert
>0\right\}  \right\vert  &  \leq c_{4}^{\prime\prime}n^{2\delta\left(
d-1\right)  }\ ,\\
\left\vert \left\{  e^{\ast}\in\left(  \overline{\partial}\mathbf{C}_{j,l}%
^{t}\left(  \omega\right)  \right)  ^{\ast}:\left\vert e^{\ast}\cap
\mathcal{H}_{\left\lfloor n\hat{x}\right\rfloor -\left\lfloor lN\hat
{x}\right\rfloor }^{t,+}\right\vert >0\right\}  \right\vert  &  \leq
c_{4}^{\prime\prime}n^{2\delta\left(  d-1\right)  }\ .
\end{align}
Therefore,
\begin{equation}
\varphi_{t}\left(  \left\lfloor n\hat{x}\right\rfloor -2\left\lfloor K_{n}%
^{N}\left(  \delta\right)  N\hat{x}\right\rfloor \right)  \leq\phi_{t}\left(
\left\lfloor n\hat{x}\right\rfloor \right)  +2c_{4}^{\prime\prime}%
n^{2\delta\left(  d-1\right)  }\ .
\end{equation}
Dividing by $n-2n^{1-\delta}$ and taking the limit for $n\rightarrow\infty,$
we obtain $\bar{\varphi}\left(  x\right)  \leq\bar{\phi}_{t}\left(  x\right)
.$
\end{proof}

\subsection{Proof of Theorem \ref{thm1}}

For any $x\in\mathbb{Z}^{d}$ and $p\in\left(  1-\frac{1}{c_{3}},1\right)  ,$
by Proposition \ref{P1}, we are left with the estimate of the probability that
there exists a finite clusters $\mathbf{C}_{\left\{  0,x\right\}  }$ such that
$\left\vert \overline{\partial}\mathbf{C}_{\left\{  0,x\right\}  }\right\vert
=\left\vert \mathbf{S}_{\left\{  0,x\right\}  }\right\vert \leq\left(
1+\delta\right)  \varphi\left(  x\right)  ,$ for $\delta$ larger than the
value $\delta^{\ast}$ given in (\ref{d*}).

Let $t\in\mathbb{S}_{x}^{d-1}.$ Moreover, given $N\in\mathbb{N}$ larger than
$1,$ let us set $K^{N}\left(  x\right)  :=\left\lfloor \frac{\left\Vert
x\right\Vert }{N}\right\rfloor -1$ and%
\begin{align}
\mathcal{H}_{i}^{t}  &  :=\mathcal{H}_{\left\lfloor iN\hat{x}\right\rfloor
}^{t}\ ;\mathcal{H}_{i}^{t,-}:=\mathcal{H}_{\left\lfloor iN\hat{x}%
\right\rfloor }^{t,-}\ ;\ \mathcal{H}_{i}^{t,+}:=\mathcal{H}_{\left\lfloor
iN\hat{x}\right\rfloor }^{t,+}\ ,\ i=0,..,K^{N}\left(  x\right)  \ ;\\
\mathcal{H}_{\left\lfloor \left(  K^{N}\left(  x\right)  +1\right)  N\hat
{x}\right\rfloor }^{t}  &  :=\mathcal{H}_{x}^{t}\ ;\ \mathcal{H}_{x}%
^{t,-}:=\mathcal{H}_{\left\lfloor \left(  K^{N}\left(  x\right)  +1\right)
N\hat{x}\right\rfloor }^{t,-}\ ;\\
\mathcal{S}_{i}^{t}  &  :=\mathcal{H}_{i}^{t,+}\cap\mathcal{H}_{i+1}^{t,-}\ .
\end{align}
With a slight notational abuse we still denote by $\mathbf{S}_{\{0,x\}}$ its
representation as a hypersurface in $\mathbb{R}^{d}$ and define
\begin{equation}
\mathbf{C}_{i}^{t}:=\mathbf{C}_{\{0,x\}}\cap\mathcal{S}_{i}^{t}\ ;\ \mathbf{S}%
_{i}^{t}:=\mathbf{S}_{\{0,x\}}\cap\mathcal{S}_{i}^{t}\ .
\end{equation}
Hence, $\mathbf{C}_{\{0,x\}}=%
{\textstyle\bigcup_{i=0}^{\left\lfloor \frac{\left\Vert x\right\Vert }%
{N}\right\rfloor -1}}
\mathbf{C}_{i}^{t}$ and $\mathbf{S}_{\{0,x\}}\cap\mathcal{S}_{\{0,x\}}%
^{t}\subseteq%
{\textstyle\bigcup_{i=0}^{\left\lfloor \frac{\left\Vert x\right\Vert }%
{N}\right\rfloor -1}}
\mathbf{S}_{i}^{t}\ .$

We say that a slab $\mathcal{S}_{i}^{t}$ is \emph{bad} if $\mathbf{S}_{i}^{t}
$ is not connected, otherwise we call it \emph{good}, and call \emph{crossing}
any connected component $\mathbf{s}$ of $\mathbf{S}_{i}^{t}$ such that,
denoting by $\mathcal{T}\left(  \mathbf{s}\right)  $ the compact subset of
$\mathcal{S}_{i}^{t}$ whose boundary is $\mathbf{s},$ there is a path in
$\mathbb{L}^{d}\cap\mathcal{T}\left(  \mathbf{s}\right)  $ connecting
$\mathcal{H}_{i}^{t}$ with $\mathcal{H}_{i+1}^{t}.$

We remark that since $\mathbf{C}_{\{0,x\}}$ is connected, the existence of two
crossings in $\mathcal{S}_{i}^{t}$ implies the existence of two disjoint paths
connecting $\mathcal{H}_{i}^{t}$ and $\mathcal{H}_{i+1}^{t}$ while the
converse does not hold true in general.

Let $\eta$ be the fraction of slabs where there are at least two crossings.
Since any crossing is composed by at least $\phi_{t}\left(  \left\lfloor
N\hat{x}\right\rfloor \right)  $ plaquettes, we have
\begin{equation}
\frac{\left\Vert x\right\Vert }{N}\left(  \eta2\phi_{t}\left(  \left\lfloor
N\hat{x}\right\rfloor \right)  +\left(  1-\eta\right)  \phi_{t}\left(
\left\lfloor N\hat{x}\right\rfloor \right)  \right)  =\frac{\left\Vert
x\right\Vert }{N}\left(  1+\eta\right)  \phi_{t}\left(  \left\lfloor N\hat
{x}\right\rfloor \right)  <\left(  1+\delta\right)  \varphi\left(  x\right)
\ .
\end{equation}
Moreover, given $\varepsilon>0,$ there exists $R_{\varepsilon}>0$ such that,
for any $x\in\mathbb{Z}^{d}\cap\left(  R_{\varepsilon}B\right)  ^{c}%
,$\linebreak$\varphi\left(  x\right)  \leq\bar{\varphi}\left(  x\right)
\left(  1+\varepsilon\right)  .$ Hence, choosing $N$ sufficiently large such
that $\phi_{t}\left(  \left\lfloor N\hat{x}\right\rfloor \right)  \leq
\bar{\phi}_{t}\left(  N\hat{x}\right)  \left(  1+\varepsilon\right)  ,$ since
$t\in\mathbb{S}_{x}^{d-1},\ \bar{\phi}_{t}\left(  N\hat{x}\right)
=\bar{\varphi}\left(  N\hat{x}\right)  $ and, by the previous inequality, we
get $\eta<\delta.$

Then, there are at most $\frac{\left\Vert x\right\Vert }{3N}-\eta
\frac{\left\Vert x\right\Vert }{N}\ 3$-tuple of consecutive slabs containing a
single crossing and therefore at most the same number of bad slabs containing
a single crossing. Choosing, $\eta\leq\frac{1}{6},$ it is possible to modify
the configuration of at most $c_{5}N^{d}$ bonds, with $c_{5}=c_{5}\left(
d,\delta\right)  ,$ inside any $3$-tuple of consecutive slabs containing a
single crossing in such a way that the resulting cluster will have at least
one $t$-bond inside each of these slabs. Since these modifications can be
performed independently, this fact and the previous proposition imply that
there exists a positive constant $c_{6}=c_{6}\left(  p\right)  $ such that
\begin{equation}
\frac{g_{t}^{\left(  p\right)  }\left(  x\right)  }{\mathbb{P}_{p}\left\{
0<\left\vert \mathbf{C}_{\{0,x\}}\right\vert <\infty\right\}  }\leq
e^{-\left\Vert x\right\Vert c_{6}}%
\end{equation}
and consequently the \emph{mass-gap} condition $f_{t}^{\left(  p\right)
}\left(  x\right)  \leq e^{-c_{7}\left\Vert x\right\Vert }h_{t}^{\left(
p\right)  }\left(  x\right)  ,\ c_{7}=c_{7}\left(  p\right)  >0,$ uniformly in
$t\in\mathbb{S}_{x}^{d-1}.$

Thus, from (\ref{renh1}) we have
\begin{equation}
\frac{\mathbb{P}_{p}\left\{  0\longleftrightarrow x\ ,\ \left\vert
\mathbf{C}_{\{0,x\}}\right\vert <\infty\right\}  }{\mathbb{P}_{p}\left\{
0\overset{h_{t}}{\longleftrightarrow}x\right\}  }\leq c_{8}\ , \label{asyeq}%
\end{equation}
with $c_{8}=c_{8}\left(  d\right)  >0,$ for any $t\in\mathbb{S}_{x}^{d-1}.$

We now proceed as in Section 4.3 of \cite{CI}. Given $t\in\mathbb{S}^{d-1},$
we extend $f_{t}^{\left(  p\right)  }$ to a function defined on the whole
lattice by setting it equal to zero where it is non defined and set
\begin{equation}
\mathbb{R}^{d}\ni s\longmapsto F_{t}^{\left(  p\right)  }\left(  s\right)
:=\sum_{x\in\mathbb{Z}^{d}}f_{t}^{\left(  p\right)  }\left(  x\right)
e^{\left\langle s,x\right\rangle }\in\overline{\mathbb{R}}\ .
\end{equation}
The renewal equation (\ref{renh}) imply
\begin{equation}
H_{t}^{\left(  p\right)  }\left(  s\right)  =\frac{1}{1-F_{t}^{\left(
p\right)  }\left(  s\right)  }\ .
\end{equation}
For $s\in\mathcal{K}_{t}^{p},$ since
\begin{equation}
\left(  s,x\right)  \leq\max_{s\in\mathcal{K}_{t}^{p}}\left(  s,x\right)
=\tau_{p}^{t}\left(  x\right)  \leq1
\end{equation}
and $h_{t}^{\left(  p\right)  }\left(  x\right)  \leq e^{-\tau_{p}^{t}\left(
x\right)  },\ F_{t}^{\left(  p\right)  }\left(  s\right)  $ is finite,
moreover it is continuous, then, $\forall s\in\partial\mathcal{K}_{t}^{p},$%
\begin{equation}
\mathbb{Z}^{d}\ni x\longmapsto q_{t;s}^{\left(  p\right)  }\left(  x\right)
:=f_{t}^{\left(  p\right)  }\left(  x\right)  e^{\left\langle s,x\right\rangle
}\in\mathbb{R}%
\end{equation}
is the density of the probability measure $Q_{t;s}^{\left(  p\right)  }$ on
$\left(  \mathbb{Z}^{d},\mathcal{B}\left(  \mathbb{Z}^{d}\right)  \right)  $
which,\ has exponentially decaying tails:%
\begin{equation}
f_{t}^{\left(  p\right)  }\left(  x\right)  e^{\left\langle s,x\right\rangle
}\leq e^{-c_{7}\left\Vert x\right\Vert }h_{t}^{\left(  p\right)  }\left(
x\right)  e^{\left\langle s,x\right\rangle }\leq e^{-c_{7}\left\Vert
x\right\Vert }\ .
\end{equation}
If $X$ is a random vector with probability distribution $Q_{t;s}^{\left(
p\right)  },$ denoting by $\mathbb{E}_{p}^{t;s}$ the expectation of a random
variable under $Q_{t;s}^{\left(  p\right)  },$ we set
\begin{equation}
\mu_{t}^{p}\left(  s\right)  :=\mathbb{E}_{p}^{t;s}\left[  X\right]
=\operatorname{grad}\log F_{t}^{\left(  p\right)  }\left(  s\right)  \ ,
\end{equation}
while
\begin{equation}
C_{t}^{p}\left(  s\right)  :=\mathrm{Hess}\log F_{t}^{\left(  p\right)
}\left(  s\right)
\end{equation}
denotes the covariance matrix of $X.$ Since $f_{t}^{\left(  p\right)  }\left(
x\right)  >0$ on a whole half-space, $C_{t}^{p}\left(  s\right)  $ is non
degenerate. Hence,
\begin{equation}
\partial\mathcal{K}_{t}^{p}=\left\{  s\in\mathbb{R}^{d}:F_{t}^{\left(
p\right)  }\left(  s\right)  =1\right\}  \subseteq\mathbb{R}^{d}%
\backslash\mathcal{D}_{t}^{p}%
\end{equation}
is a real analytic strictly convex surface with Gaussian curvature uniformly
bounded away from zero and therefore, because $Q_{t;s}^{\left(  p\right)  }$
is supported on $\mathcal{H}_{0}^{t,+}\cap\mathbb{Z}^{d},\ \mu_{t}^{p}\left(
s\right)  \neq0$ and $\left(  s,\mu_{t}^{p}\left(  s\right)  \right)  >0$ for
any $s\in B_{r}\left(  t\right)  \cap\partial\mathcal{K}_{t}^{p}$ with $r$
sufficiently small.

Let then $s\in B_{r}\left(  t\right)  \cap\partial\mathcal{K}_{t}^{p},$ for
any $\mu\in B\left(  \mu_{t}^{p}\left(  s\right)  \right)  \cap\mathcal{H}%
_{\mu_{t}^{p}\left(  s\right)  }^{s},$ if $\left\{  X_{i}\right\}  _{i\geq1}$
is a sequence of i.i.d. random vectors distributed according to $Q_{t;s}%
^{\left(  p\right)  },$ for $n\in\mathbb{N},$ we can rewrite (\ref{renh}) as
\begin{equation}
h_{t}^{\left(  p\right)  }\left(  \left\lfloor n\mu\right\rfloor \right)
=\delta_{0}\left(  \left\lfloor n\mu\right\rfloor \right)  +e^{-\left\langle
\left\lfloor n\mu\right\rfloor ,s\right\rangle }\sum_{k\geq1}\bigotimes
_{i=1}^{k}Q_{t;s}^{\left(  p\right)  }\left\{  \sum_{i=1}^{k}X_{i}%
=\left\lfloor n\mu\right\rfloor \right\}  \ . \label{renh2}%
\end{equation}
Hence, there exist two positive constant $c_{9}=c_{9}\left(  p\right)  $ and
$c_{10}=c_{10}\left(  p\right)  $ such that
\begin{equation}
\left\Vert n\mu-\sum_{i=1}^{k}\mathbb{E}_{p}^{s;t}\left[  X_{i}\right]
\right\Vert =\left\Vert n\mu-k\mu_{t}^{p}\left(  s\right)  \right\Vert \geq
nc_{9}\left\Vert \mu-\mu_{t}^{p}\left(  s\right)  \right\Vert +\left\vert
n-k\right\vert c_{10}\left\Vert \mu_{t}^{p}\left(  s\right)  \right\Vert \ .
\end{equation}
Therefore, the standard large deviation upper bound for
\begin{equation}
\bigotimes_{i=1}^{k}Q_{t;s}^{\left(  p\right)  }\left\{  \sum_{i=1}^{k}%
X_{i}=\left\lfloor n\mu\right\rfloor \right\}  \leq e^{-c_{11}\frac{\left(
n-k\right)  }{k}^{2}\wedge\left\vert n-k\right\vert -c_{12}\frac{n^{2}}%
{k}\wedge n\left\Vert \mu-\mu_{t}^{p}\left(  s\right)  \right\Vert ^{2}}%
\end{equation}
leads to
\begin{equation}
h_{t}^{\left(  p\right)  }\left(  \left\lfloor n\mu\right\rfloor \right)  \leq
c_{13}\sqrt{n}e^{-\left[  \left\langle \left\lfloor n\mu\right\rfloor
,s\right\rangle +c_{14}n\left\Vert \mu-\mu_{t}^{p}\left(  s\right)
\right\Vert ^{2}\right]  } \label{LDh}%
\end{equation}
with $c_{13}=c_{13}\left(  p\right)  ,\ c_{14}=c_{14}\left(  p\right)  $
positive constants. Then, (\ref{tau}) and (\ref{LDh}) imply
\begin{equation}
\tau_{p}^{t}\left(  \mu\right)  \geq\left\langle \mu,s\right\rangle
+c_{14}\left\Vert \mu-\mu_{t}^{p}\left(  s\right)  \right\Vert ^{2}\ ,
\label{lbtau}%
\end{equation}
that is the strict convexity of $\tau_{p}^{t}.$ Moreover, because $\tau
_{p}^{t},$ being an equivalent norm in $\mathbb{R}^{d},$ is lower
semicontinuous, from (\ref{lbtau}) it follows that $\mu_{t}^{p}\left(
s\right)  $ and $s$ are in polar relation with respect to $\tau_{p}^{t},$
namely $\left\langle s,\mu_{t}^{p}\left(  s\right)  \right\rangle =\tau
_{p}\left(  \mu_{t}^{p}\left(  s\right)  \right)  .$

Furthermore, for any $x\in\mathbb{Z}^{d},$ let $t\in\mathbb{S}_{x}^{d-1}.$
There exist $r^{\prime}>0,\ s\in B_{r^{\prime}}\left(  t\right)  \cap
\partial\mathcal{K}_{t}^{p},$ and $n_{x}\in\mathbb{N}$ such that
\begin{equation}
\left\Vert x-n_{x}\mu_{t}^{p}\left(  s\right)  \right\Vert \leq c_{15}\ ,
\label{x-nmu}%
\end{equation}
where $c_{15}=c_{15}\left(  p\right)  >0.$ Hence, choosing $\alpha\in\left(
0,\frac{1}{2}\right)  ,$ since $Q_{t;s}^{\left(  p\right)  }$ is a centered
lattice distribution, we can apply the local central limit theorem for
$\bigotimes_{i=1}^{k}Q_{t;s}^{\left(  p\right)  }\left\{  \sum_{i=1}^{k}%
X_{i}=x\right\}  $ to get,
\begin{align}
&  \sum_{k\in\mathbb{N}\ :\ \left\vert k-n_{x}\right\vert <n_{x}^{\frac{1}%
{2}+\alpha}}\bigotimes_{i=1}^{k}Q_{t;s}^{\left(  p\right)  }\left\{
\sum_{i=1}^{k}X_{i}=x\right\} \\
&  =\sum_{k\in\mathbb{N}\ :\ \left\vert k-n_{x}\right\vert <n_{x}^{\frac{1}%
{2}+\alpha}}\frac{\exp\left\{  -\frac{\left(  n_{x}-k\right)  ^{2}}%
{2k}\left\langle \left(  C_{t}^{p}\right)  ^{-1}\left(  s\right)  \mu_{t}%
^{p}\left(  s\right)  ,\mu_{t}^{p}\left(  s\right)  \right\rangle \right\}
}{\sqrt{2\pi k\det C_{t}^{p}\left(  s\right)  }}\left(  1+o\left(  1\right)
\right) \nonumber\\
&  =\sum_{k\in\mathbb{N}\ :\ \left\vert k-n_{x}\right\vert <n_{x}^{\frac{1}%
{2}+\alpha}}\frac{\exp\left\{  -\frac{\left(  n_{x}-k\right)  ^{2}}%
{2n_{x}\left(  1+O\left(  n_{x}^{\alpha-\frac{1}{2}}\right)  \right)
}\left\langle \left(  C_{t}^{p}\right)  ^{-1}\left(  s\right)  \mu_{t}%
^{p}\left(  s\right)  ,\mu_{t}^{p}\left(  s\right)  \right\rangle \right\}
}{\sqrt{2\pi n_{x}\left(  1+O\left(  n_{x}^{\alpha-\frac{1}{2}}\right)
\right)  \det C_{t}^{p}\left(  s\right)  }}\left(  1+o\left(  1\right)
\right) \nonumber\\
&  =\frac{1}{\sqrt{\left(  2\pi n_{x}\right)  ^{d-1}\det C_{t}^{p}\left(
s\right)  \left\langle \left(  C_{t}^{p}\right)  ^{-1}\left(  s\right)
\mu_{t}^{p}\left(  s\right)  ,\mu_{t}^{p}\left(  s\right)  \right\rangle }%
}\left(  1+o\left(  1\right)  \right)  \ .\nonumber
\end{align}
On the other hand, in the complementary range of $k$'s, proceeding as in
(\ref{LDh}) we obtain
\begin{equation}
\sum_{k\in\mathbb{N}\ :\ \left\vert k-n_{x}\right\vert \geq n_{x}^{\frac{1}%
{2}+\alpha}}\bigotimes_{i=1}^{k}Q_{t;s}^{\left(  p\right)  }\left\{
\sum_{i=1}^{k}X_{i}=x\right\}  \leq e^{-c_{15}n_{x}^{2\alpha}}%
\end{equation}
which, by (\ref{x-nmu}), (\ref{renh2}) and (\ref{asyeq}), since $s=s\left(
\hat{x}\right)  ,$ gives back
\begin{equation}
\mathbb{P}_{p}\left\{  0\longleftrightarrow x\ ,\ \left\vert \mathbf{C}%
_{\{0,x\}}\right\vert <\infty\right\}  =\frac{\Phi_{p}\left(  \hat{x}\right)
}{\sqrt{\left(  2\pi\left\Vert x\right\Vert \right)  ^{d-1}}}e^{-\tau_{p}%
^{t}\left(  x\right)  }\left(  1+o\left(  1\right)  \right)  \ ,
\end{equation}
where
\begin{equation}
\mathbb{S}^{d-1}\ni\hat{x}\longmapsto\Phi_{p}\left(  \hat{x}\right)
:=\sqrt{\frac{\left\Vert \mu_{t}^{p}\left(  s\left(  \hat{x}\right)  \right)
\right\Vert ^{d-1}}{\det C_{t}^{p}\left(  s\left(  \hat{x}\right)  \right)
\left\langle \left(  C_{t}^{p}\right)  ^{-1}\left(  s\left(  \hat{x}\right)
\right)  \mu_{t}^{p}\left(  s\left(  \hat{x}\right)  \right)  ,\mu_{t}%
^{p}\left(  s\left(  \hat{x}\right)  \right)  \right\rangle }}\in
\mathbb{R}^{+}\ .
\end{equation}
is a real analytic function on $\mathbb{S}^{d-1}.$

\end{document}